\documentclass[reqno]{amsart}
\usepackage{amssymb,amsfonts,amstext,amsthm,hyperref,cleveref,xcolor,amsmath} 

\theoremstyle{plain}
\newtheorem{thm}{Theorem}[section]
\newtheorem{lem}[thm]{Lemma}
\newtheorem{cor}[thm]{Corollary}
\newtheorem{prop}[thm]{Proposition}

\theoremstyle{definition}
\newtheorem{defn}[thm]{Definition} 

\newtheorem{rem}[thm]{Remark}

\theoremstyle{remark}

\numberwithin{equation}{section}

\usepackage{hyperref}
\usepackage{relsize}
\usepackage{dsfont}
\usepackage{multirow}
\usepackage{comment}

\def\e{\mathbf{e}}
\def\t{\mathbf{t}}
\def\F{\mathbb{F}}

\def\tr{\mathrm{tr}}

\newcommand{\ds}{\displaystyle}

\newcommand{\rvline}{\hspace*{-\arraycolsep}\vline\hspace*{-\arraycolsep}}


\begin{document}
	\title{Sum of two strictly $n$-zero matrices}
		\author[I. G. Gargate]{Ivan G. Gargate}
	\address{UTFPR, Campus Pato Branco, Rua Via do Conhecimento km 01, 85503-390 Pato Branco, PR, Brazil}
	\email{ivangargate@utfpr.edu.br}
	\author[E.~Saenz]{Edgar A. Saenz}	
\address{Department of Mathematics\\ Virginia Tech\\
Blacksburg, VA 24061\\ U.S.A.}
\email{easaenzm@math.vt.edu}

\maketitle
	
	\begin{abstract} In this work we investigate when an $n\times n$ Jordan matrix can be written as the sum of two strictly $n$-zero matrices. In particular, we show that if $\mathbb{F}$ is an algebraically closed field of characteristic zero and $A$ is an $n\times n$ matrix over $\mathbb{F}$ with $\tr(A)=0$, then $A$ is the sum of two strictly $n$-zero matrices. 
	\end{abstract}
	\maketitle
	\section{Introduction}
	Throughout this article, $\F$ denotes a field and $M_{n}(\F)$ denotes the algebra of $n\times n$ matrices over $\F$. We say that $A\in M_{n}(\F)$ is strictly $k$-zero if $A$ is nilpotent with index of nilpotency $k$. If $k\in\{1,2\}$, we simply say that $A$ is a square-zero matrix.
	
	The problem of expressing a square matrix as a sum of finitely many strictly $k$-zero matrices have been of interest in the literature (see [\cite{Botha},\cite{Paras},\cite{Pazzis},\cite{Takahashi},\cite{Wang}]). Much of the research has been on sums of square-zero matrices. For instance,
	 Wang and Wu \cite{Wang} showed that a trace-zero square matrix $A$ over the field of complex numbers is the sum of two square-zero matrices if and only if $A$ is similar to a matrix of the form $S\oplus(-S)\oplus N$, where $S$ is nonsingular and $N$ is nilpotent. In \cite{Botha}, Botha generalized the results given in \cite{Wang} and classified the matrices over any given field that are the sum of two square-zero matrices. In \cite{Pazzis}, de Seguins Pazzis showed that if $\text{char}(\F)=2$, then every trace-zero matrix of $M_{n}(\F)$ is the sum of three square-zero matrices. 
	 
	On the other hand, Fillmore's theorem \cite{Fillmore} says that if $A$ is a nonscalar matrix in $M_{n}(\F)$, then the trace of $A$ is zero if and only if $A$ is the sum of two nilpotent matrices (possibly with distinct index of nilpotency). Using Fillmore's theorem, it has been shown by Monterde and Paras \cite{Paras} that if $\F$ has at least three elements, then every trace-zero $n\times n$ nonscalar matrix over $\F$ is the sum of four strictly $k$-zero matrices for all $k\in\{2,\dots,n\}$.
	
	In this article, motivated by the algorithm given by Borobia \cite{Bo}, 
	we show that if $J\in M_{n}(\F)$ is a nonsingular Jordan matrix  with at least two distinct eigenvalues, zero trace, and whose spectrum contains neither $1$ nor $-1$ if $J$ is nondiagonal, then $J$ is similar to a sum of the form $U+L$ where $U$ is a strictly upper triangular matrix with nonzero entries on its superdiagonal and $L$ is a strictly lower triangular matrix with nonzero entries on its subdiagonal. This result implies that such a matrix $J$ can be written as the sum of two $n$-strictly zero matrices. Our proof is constructive and relies only on elementary results, so a pair of matrices $U$ and $L$ can be computed explicitly for any given $J$ that verifies the assumptions above (see Theorem \ref{thm:1}). 
	
	We also show that if  $\F$ is an algebraically closed field of characteristic zero and $A\in M_{n}(\F)$, then $\tr(A)=0$ is a necessary and sufficient condition to express the matrix $A$ as the sum of two $n$-zero matrices  (see Theorem \ref{thm:2}). In this particular situation, we have improved the result given by Monterde and Para \cite{Paras}, in the case $k=n$, by reducing the number of terms in the sum from four to only two.

	\section{Notations and preliminaries}
	We start with some notations, definitions and standard results. In the following
	
	\begin{itemize}
		\item $\text{char}(\mathbb{F})$ is the characteristic of the field $\mathbb{F}$,
		\item $\tr(A)$ is the trace of a square matrix $A$,
		\item $\mathbf{0}$ denotes the zero matrix (regardless of the dimensions), 
		\item $\F^{n}$ is the set of all $n\times 1$ column vectors over $\mathbb{F}$,
		\item $\mathrm{diag}(A_{1},A_{2},\dots,A_{r})$ denotes a block diagonal matrix,
		\item $\e_{n}$ is the column vector in $\F^{n}$ whose entries are all equal to $1$,
		\item $A^{\t}$ is the transpose of the matrix $A$,
		\item the $n\times n$ identity matrix is denoted by $I_{n}$ or simply by $I$ if the size is not emphasized.
	\end{itemize}
	
	\begin{defn}\label{def:ul} A matrix $A\in M_{n}(\mathbb{F})$ is called scalar if $A=\lambda I_{n}$ for some $\lambda\in\mathbb{F}$.
	\end{defn} 
	
	\begin{defn} Let $\lambda\in\mathbb{F}$. The $k\times k$ Jordan block over $\mathbb{F}$ corresponding to $\lambda$ is defined as $J_{1}(\lambda)=[\lambda]$ for $k=1$ and 
		$$J_{k}(\lambda)=\left[\begin{array}{ccccccc} 
			\lambda & 1       & 0       & \cdots  &      0 & 0 \\ 
			0   & \lambda & 1       & \cdots  &      0 & 0 \\
			0   &  0      & \lambda & \dots  &      0 & 0 \\
			\vdots & \vdots  & \vdots  & \ddots  & \ddots & \vdots\\ 
			0   & 0       & 0       & \cdots  &\lambda & 1\\
			0   & 0       & 0       & \cdots &0 & \lambda\end{array} \right]_{k\times k}$$
		for $k>1$. Note that $J_{k}(0)$ is a strictly $k$-zero matrix.
	\end{defn}
	
	\begin{defn} Let $\mathbb{F}$ be a field. A Jordan matrix is any block diagonal matrix whose blocks are Jordan blocks over $\mathbb{F}$. 
		
		Up to permutation of the Jordan blocks any $n\times n$ Jordan matrix $J$ over $\mathbb{F}$ has the form
		\begin{equation}\label{eq:1}
			J=\mathrm{diag}(J(\lambda_1,m_{1}), J(\lambda_2,m_{2}),\dots, J(\lambda_r,m_{r}))
		\end{equation}
		where 
		\begin{itemize}
			\item $\lambda_{1},\lambda_{2},\dots,\lambda_{r}$ are distinct elements of $\mathbb{F}$,
			\item $m_{1},m_{2},\dots,m_{r}$ are positive integers such that $\ds\sum_{i=1}^{r}m_{i}=n$,
			\item and $$J(\lambda_{j},m_{j})=\left[\begin{array}{cccc}
				J_{n_{j,1}}(\lambda_j) & \mathbf{0} & \cdots & \mathbf{0} \\ 
				\mathbf{0} & J_{n_{j,2}}(\lambda_j) & \cdots & \mathbf{0} \\ 
				\vdots & \vdots & \ddots & \vdots \\
				\mathbf{0} & \mathbf{0} & \cdots & J_{n_{j,\mathsmaller{k_{j}}}}(\lambda_j)
			\end{array}\right]$$
			with $n_{j,\mathsmaller{1}}\geq n_{j,\mathsmaller{2}}\geq\cdots\geq n_{j,k_{j}}\geq 1$ and 
			$n_{j,\mathsmaller{1}}+n_{j,\mathsmaller{2}}+\cdots+ n_{j,\mathsmaller{k_{j}}}=m_{j}$.
		\end{itemize}
	\end{defn}
	
	\begin{rem}\label{rem:1} In (\ref{eq:1}), the elements $\lambda_{1},\lambda_{2},\dots,\lambda_{r}$ are the distinct eigenvalues of $J$, and the integer $m_{j}$ is the algebraic multiplicity of the eigenvalue $\lambda_{j}$ for $j=1,\dots,r$.
	\end{rem}\label{rem:1}
	
	We now summarize three elementary lemmas that will be useful in the proof of our main results. An interested reader may find Lemma \ref{lem:1} in Section 2 of \cite{Liv}. The other two lemmas can be demonstrated by routine computations.
	
	\begin{lem}\label{lem:1} Let $\lambda\in\mathbb{F}$, let $A\in M_{s}(\mathbb{F})$, and let $N\in M_{r}(\mathbb{F})$. If $\lambda$ is not an eigenvalue of $A$ and $N$ is nilpotent, then the two matrices 
		
		$$\left[\begin{array}{cc} 
			\lambda I+N &  E \\ 
			\mathbf{0} & A\end{array} \right]
		\ \ \ \text{and}\ \ \ \ \ \left[\begin{array}{cc} 
			\lambda I+N &  \mathbf{0} \\ 
			\mathbf{0} & A\end{array} \right]$$
		are similar (for any choice of the $r\times s$ matrix $E$ with entries en $\F$).
	\end{lem} 
	
	\begin{lem}\label{lem:2} Let $\lambda\in\F$ and $m\geq 2$. Suppose that 
		$$P=\left[\begin{array}{cc} 1 & \mathbf{0} \\ 
			-\lambda\e_{m-1} & I_{m-1}
		\end{array}\right]\ \
		\text{and}\ \ \ J=\mathrm{diag}(J_{n_{1}}(\lambda),\cdots,J_{n_{k}}(\lambda))$$
		with $n_{\mathsmaller{1}}\geq n_{\mathsmaller{2}}\geq\cdots\geq n_{k}\geq 1$ and $n_{1}+n_{2}+\cdots+ n_{k}=m$. 
		\begin{itemize}
			\item[(i)] If $m=n_{1}=2$, then
			$$PJP^{-1}=\left[\begin{array}{cc} 2\lambda & 1 \\ 
				-\lambda^{2} &  0\end{array}\right].$$
			\item[(ii)] If $m\geq 3$ and $n_{1}\geq2$ then
			$$PJP^{-1}=\left[\begin{array}{cccccc} 
				2\lambda & 1 & \rvline & \begin{array}{cc}
					0& \mathbf{0}\end{array} \\ 
				\lambda y-\lambda^{2}&  0 & \rvline & \begin{array}{cc}
					y& \mathbf{0}\end{array} \\
				\hline
				\ast& -\lambda\e_{m-2} & \rvline & J'\end{array} \right],$$
			where $y$ is the element in the $(2,3)$ entry of $J$ and $J'$ is the submatrix of $J$ obtained by removing the first two rows and the first two columns of $J$. 
		\end{itemize}
	\end{lem}
	
	\begin{lem}\label{lem:3} Let $L\in M_{n}(\mathbb{F})$. If $L$ is a strictly lower triangular matrix, then $L^{n-1}\e_{n}=(0,\cdots,0,w)^\t$ where $w$ is the product of the subdiagonal entries of $L$.\end{lem} 
	
	\section{The $(U,L)$-property}
	We now introduce a definition and discuss some basic facts that will be relevant in the next section. 
	\begin{defn}\label{def:ul} Let $A\in M_{n}(\mathbb{F})$. We say that $A$ verifies the $(U,L)$-property if $A$ is similar to a matrix $G$ whose diagonal entries are all equal to zero, and whose superdiagonal and subdiagonal entries are nonzero.
	\end{defn} 
	
	\begin{prop}\label{prop:1} Let $A\in M_{n}(\mathbb{F})$ and suppose that $A$ verifies the $(U,L)$-property. Then the following statements hold.
		\begin{itemize}
			\item[(i)] $\tr(A)=0$.
			\item[(ii)] If $t\in\mathbb{F}\setminus\{0\}$, then $tA$ verifies the $(U,L)$-property.
			\item[(iii)] If $B$ is similar to $A$, then $B$ verifies the $(U,L)$-property.
			\item[(iv)] $A$ is the sum of two strictly $n$-zero matrices.
		\end{itemize}	
	\end{prop}
	\begin{proof}
		Statements (i)-(iii) are clear. To prove statement (iv), note that there is a strictly upper triangular matrix $U$ with nonzero entries on its superdiagonal, a strictly lower triangular matrix $L$ with nonzero entries on its subdiagonal, and a nonsingular matrix $X$ such that $XAX^{-1}=U+L$. The matrices $U$ and $L$ are strictly $n$-zero matrices and $A=X^{-1}UX+X^{-1}LX$. Since the index of nilpotency is invariant under similarity, $X^{-1}UX$ and $X^{-1}LX$ are both strictly $n$-zero matrices.
	\end{proof}
	
	\begin{rem} It is clear that every $2\times 2$ matrix that verifies the $(U,L)$-property is nonsingular. So there are matrices of trace zero that can be written as the sum of two strictly 2-zero matrices but that do not verify the $(U,L)$-property. For instance, consider $A=\left[\begin{smallmatrix}
			0&1\vspace{.01in}\\
			0&0
		\end{smallmatrix}\right]$
		over a field $\mathbb{F}$ with at least $3$ elements, and let $\tau\in\mathbb{F}\setminus\{0,1\}$. If  $B=\left[\begin{smallmatrix} 
			0 &  1-\tau \\ 
			0 & 0\end{smallmatrix} \right]$ and  $C=\left[\begin{smallmatrix} 
			0 &  \tau \\ 
			0 & 0\end{smallmatrix} \right]$, then $A=B+C$ and the matrices $B$ and $C$ are strictly 2-zero matrices; however, $A$ does not verify the $(U,L)$-property.  	
	\end{rem}
	
	\begin{thm}\label{thm:jo}
		Let $A'\in M_{n}(\mathbb{F})$ be a matrix that verifies the $(U,L)$-property and let $J(0,m)\in M_{m}(\mathbb{F})$ be any Jordan matrix whose only eigenvalue is zero. If there exists an element $a\in\F\setminus\{0\}$ such that $J(0,m)+aJ_{m}(0)$ is a strictly $m$-zero matrix, then $A=\mathrm{diag}(A',J(0,m))$ is the sum of two strictly $(n+m)$-zero matrices.	
	\end{thm}

	\begin{proof} It suffices to show that the statement holds for some matrix similar to $A$. By Definition \ref{def:ul}, there is a strictly upper triangular matrix $U$ with nonzero entries on its superdiagonal, a strictly lower triangular matrix $L$ with nonzero entries on its subdiagonal, and a nonsingular matrix $X$ such that $XA'X^{-1}=U+L$.
	
		Let $N=J(0,m)+aJ_{m}(0)$, let $B=\mathrm{diag}(X,I_{m})$, and let $Y$ be the $n\times m$ matrix whose first column is $\e_{n}$ and whose entries in any other column are equal to zero. Then $BAB^{-1}=\mathbf{P}+\mathbf{Q}$, where
		$$\mathbf{P}=\left[\begin{array}{cccc} 
			U & Y \\ 
			\mathbf{0}& N\end{array} \right]\ \ \text{and}\ \ \ \mathbf{Q}=\left[\begin{array}{cccc} 
			L & -Y \\ 
			\mathbf{0}& -aJ_{m}(0)\end{array} \right].$$
	
		Note that $N=J_{1}(0)=[\,0\,]$ if $m=1$, and that $N$ is a strictly upper triangular matrix with nonzero entries on its superdiagonal if $m\geq2$. Combining this and the definition of $Y$, it follows that $\mathbf{P}$ is a strictly upper triangular matrix with nonzero entries on its superdiagonal. Hence $\mathbf{P}$ is a strictly $(n+m)$-zero matrix. Because $L$ and $-J_{m}(0)$ are nilpotent, $\mathbf{Q}$ also is. To reach the desired conclusion it suffices to show that the index of nilpotency of $-\mathbf{Q}$ is precisely $n+m$. Let $\mathbf{v}$ be the column vector $(0,\dots,0,1)^{\t}\in\mathbb{F}^{n+m}$. A straightforward calculation yields  
		$$(-\mathbf{Q})^{m}\mathbf{v}=a^{m-1}\left[\begin{array}{c} 
			\e_{n} \\ 
			\mathbf{0}\end{array} \right].$$
		Then
		$$(-\mathbf{Q})^{n+m-1}\mathbf{v}=a^{m-1}(-\mathbf{Q})^{n-1}\left[\begin{array}{c} 
			\e_{n} \\ 
			\mathbf{0}\end{array} \right]
		=a^{m-1}\left[\begin{array}{c} 
			(-L)^{n-1}\e_{n} \\ 
			\mathbf{0}\end{array} \right].$$
		Since $-L$ is a strictly lower triangular $n\times n$ matrix with nonzero entries on its subdiagonal, by Lemma \ref{lem:3}, it follows that $(-\mathbf{Q})^{n+m-1}\mathbf{v}$ is a nonzero column vector. Thus, $\mathbf{Q}$ is a strictly $(n+m)$-zero matrix.
	\end{proof}
	
	\begin{cor}\label{cor:jo} Let $A'\in M_{n}(\mathbb{F})$ be a matrix that verifies the $(U,L)$-property and let $J(0,m)\in M_{m}(\mathbb{F})$ be any Jordan matrix whose only eigenvalue is zero. If $\text{char}(\mathbb{F})\neq 2$, or if $\text{char}(\mathbb{F})=2$ and $\mathbb{F}$ has at least $4$ elements, then the matrix $A=\mathrm{diag}(A',J(0,m))$ is the sum of two strictly $(n+m)$-zero matrices.
	\end{cor}
	\begin{proof} If $m=1$, then $J(0,1)+aJ_{1}(0)$ is the strictly $1$-zero matrix for any $a\in\F$. In this case, the conclusion follows directly from Theorem \ref{thm:jo} regardless of the characteristic of $\F$.
		
	In this paragraph we assume that $m\geq2$. If $\text{char}(\mathbb{F})\neq 2$, consider $a=1$. If  $\text{char}(\mathbb{F})=2$ and $\mathbb{F}$ has at least $4$ elements, fix an element $a\in\mathbb{F}\setminus\{0,1\}$. In either case, $a\neq0$ and $a+1\neq 0$. Under these settings, $J(0,m)+aJ_{m}(0)$ is a strictly $m$-zero matrix. The conclusion now follows from Theorem \ref{thm:jo}.
	\end{proof}	
	
	\section{Main Results}
	Our first result is a stronger version of Filmore's Theorem restricted to a suitable collection of nonsingular Jordan matrices with trace zero.
	
	\begin{thm}\label{thm:1} Let $\mathbb{F}$ be a field and let $J\in M_{n}(\mathbb{F})$ be a nonsingular Jordan matrix with at least two distinct eigenvalues and such that $\tr(J)=0$. If $J$ is a nondiagonal matrix whose spectrum contains neither $1$ nor $-1$, or if $J$ is a diagonal matrix, then $J$ verifies the $(U,L)$-property. 
	\end{thm} 
	
	\begin{proof} Under the given assumptions, we may and do assume that $J$ is as in (\ref{eq:1}) with $r\geq2$. Let $H$ be the $n\times n$ matrix given by
		$$H=\left[\begin{array}{cc} 
			J_{1} & E\\
			\mathbf{0} & \tilde{J}\end{array}\right],$$
		where $J_{1}=J(\lambda_{1},m_{1})$, $\tilde{J}=\mathrm{diag}(J(\lambda_{2},m_{2}),\dots,J(\lambda_{r},m_{r}))$, and $E=\left[\begin{array}{c} \e_{n-m_{1}}^{\t} \\ 
			\mathbf{0}
		\end{array}\right]$.
		By Lemma \ref{lem:1} $J$ and $H$ are similar. By statement (ii) of Proposition \ref{prop:1}, it suffices to show that $H$ verifies the $(U,L)$-property. To do so, we proceed by cases.
		
		\textbf{Case 1:} $J$ is a nondiagonal matrix whose spectrum contains neither $1$ nor $-1$. In this case, $r\geq2$ and there exists $j_{0}$ such that $n_{j_{0},1}\geq2$. By similarity we may assume that $j_{0}=1$. For the rest of the proof of this case, $n_{1,1}\geq 2$ and so $m_{1}\geq 2$. 
		
		Let $B$ be equal to $I_{n}$ except on its column 1 and on its row $m_{1}+1$. We define column 1 of $B$ as the transpose of the row vector  
		$$(1,\underbrace{-\lambda_{1},\cdots,-\lambda_{1}}_{m_{1}-1},\underbrace{-\lambda_{2},\cdots,-\lambda_{2}}_{m_{2}},\cdots,\underbrace{-\lambda_{r},\cdots,-\lambda_{r}}_{m_{r}}),$$
		and row $m_{1}+1$ of $B$ as
		$$(-\lambda_{2},0,\underbrace{-1,\cdots,-1}_{m_{1}-2},1,\underbrace{0,\cdots,0}_{n-m_{1}-1}).$$
		More precisely, $$B=\left[\begin{array}{cc} P & \mathbf{0} \\ S & I_{n-m_{1}} 
		\end{array}\right],$$
		where
		$$P=\left[\begin{array}{cc} 1 & \mathbf{0} \\ 
			-\lambda_{1}\e_{m_1-1} & I_{m_{1}-1}
		\end{array}\right]\ \ \
		\text{and}\ \ \ 
		S=\left[\begin{array}{cccc} 
			-\lambda_{2}&0 & -\e_{m_{1}-2}^{\t}\\
			-\lambda_{2}\e_{m_{2}-1} & \mathbf{0} & \mathbf{0}\\
			-\lambda_{3}\e_{m_{3}} &  \mathbf{0} & \mathbf{0}\\
			\vdots & \vdots & \vdots\\
			-\lambda_{r}\e_{m_{r}} & \mathbf{0} & \mathbf{0}\\
		\end{array}\right].$$
		So, $$B^{-1}=\left[\begin{array}{cc} P^{-1} & \mathbf{0} \\
			T & I_{n-m_{1}} 
		\end{array}\right],$$
		where
		$$P^{-1}=\left[\begin{array}{cc} 1 & \mathbf{0} \\ 
			\lambda_{1}\e_{m_{1}-1} & I_{m_{1}-1}
		\end{array}\right]\ \ \ 
		\text{and}\ \ \ T=\left[\begin{array}{cccc} 
			(m_{1}-2)\lambda_{1}+\lambda_{2}&0 & \e_{m_{1}-2}^{\t}\\
			\lambda_{2}\e_{m_{2}-1} & \mathbf{0} & \mathbf{0}\\
			\lambda_{3}\e_{m_{3}} &  \mathbf{0} & \mathbf{0}\\
			\vdots & \vdots & \vdots\\
			\lambda_{r}\e_{m_{r}} & \mathbf{0} & \mathbf{0}\\
		\end{array}\right].$$
		Then 
		$$BHB^{-1}=\left[\begin{array}{ccc} PJ_{1}P^{-1}+PET & PE\\
			SJ_{1}P^{-1}+(SE+\tilde{J})T &  SE+\tilde{J}\end{array}\right].$$
		
		We next determine by blocks the only possible diagonal, superdiagonal, and subdiagonal entries of $BHB^{-1}$.
		\begin{itemize}
			\item[(1)]\label{i:1} By Lemma \ref{lem:2}, 
			either 
			$$PJ_{1}P^{-1}=\left[\begin{array}{cccccc} 
				\ \ 2\lambda_{1} & 1 \\ 
				-\lambda_{1}^{2}&  0 \end{array} \right]$$
			or
			$$PJ_{1}P^{-1}=\left[\begin{array}{cccccc} 
				2\lambda_{1} & 1 & \rvline & \begin{array}{cc}
					0& \mathbf{0}\end{array} \\ 
				\lambda_{1}y-\lambda_{1}^{2}&  0 & \rvline & \begin{array}{cc}
					y& \mathbf{0}\end{array} \\
				\hline
				\ast& -\lambda_{1}\e_{m-2} & \rvline & J_{1}'\end{array} \right],$$
			where $y$ is the element in the $(2,3)$ entry of $J_{1}$ and $J'_{1}$ is the submatrix of $J_1$ obtained by removing the first two rows and the first two columns of $J_{1}$. On the other hand, using the fact that $\tr(J)=0$ one can easily get $ET$. Moreover,
			$$P(ET)=P\left[\begin{array}{ccc}
				-2\lambda_{1} & 0 & \e^{\t}_{m_{1}-2} \\
				\mathbf{0} & \mathbf{0} & \mathbf{0}\\
			\end{array}\right]
			=\left[\begin{array}{cccc}
				-2\lambda_{1} & 0 &\ \ \ \ \ \e^{\t}_{m_{1}-2}\\
				\ \ 2\lambda_{1}^2 & 0 & -\lambda_{1}\e^{\t}_{m_{1}-2}\\
				\vdots & \vdots & \vdots\\
				\ \ 2\lambda_{1}^2 & 0 & -\lambda_{1}\e^{\t}_{m_{1}-2}\\
			\end{array}\right].$$
			So, with respect to the matrix $PJ_{1}P^{-1}+PET$ the following holds: its diagonal entries are all equal to zero, its superdiagonal entries are contained in the set $\{1,1-\lambda_{1},-\lambda_{1}\}$, and because $y\in\{0,1\}$, its subdiagonal entries are contained in the set $\{\lambda_{1}^{2},\lambda_{1}+\lambda_{1}^{2},-\lambda_{1}\}$.
			\item[(2)] The entry in the lower left corner of $PE$ is $-\lambda_{1}$.
			
			\item[(3)] $SE$ is an $(n-m_{1})\times(n-m_{1})$ matrix. Moreover,  $$SE+\tilde{J}=-\left[\begin{array}{cccc} 
				\lambda_{2}\e_{m_{2}}& \lambda_{2}\e_{m_{2}}&\cdots& \lambda_{2}\e_{m_{2}}\\
				\lambda_{3}\e_{m_{3}}& \lambda_{3}\e_{m_{3}}&\cdots& \lambda_{3}\e_{m_{3}}\\
				\vdots & \vdots & \vdots & \vdots\\
				\lambda_{r}\e_{m_{r}}& \lambda_{r}\e_{m_{r}}&\cdots& \lambda_{r}\e_{m_{r}}\\
			\end{array}\right]+\tilde{J}.$$
			So, with respect to the matrix $SE+\tilde{J}$ the following holds: its diagonal entries are all equal to zero, its superdiagonal entries are contained in the set $\{1-\lambda_{j},-\lambda_{j}: 2\leq j\leq r\}$, and its subdiagonal entries are contained in the set $\{-\lambda_{j}: 2\leq j\leq r\}$. 
			
			\item[(4)] From (3) and the definition of $T$, it follows that the entry in the upper right corner of $(SE+\tilde{J})T$ is zero. On the other hand, the last column of $J_{1}P^{-1}$ coincides with the last column of $J_{1}$. So the entry in the upper right corner of  $SJ_{1}P^{-1}+(SE+\tilde{J})T$ is either $-\lambda_{2}$, or $-\lambda_1$, or $-1-\lambda_{1}$. 
		\end{itemize}
		Combining $(1)$-$(4)$ and the fact that the spectrum of $J$ does not contain any element of the set $\{0,1,-1\}$,
		it follows that $G=BHB^{-1}$ is a matrix whose diagonal entries are all equal to zero, and whose superdiagonal and subdiagonal entries are nonzero. Thus, $H$ verifies the $(U,L)$-property.
		
	\textbf{Case 2:} $J$ is a diagonal matrix. 
	
	In this case, $J=\mathrm{diag}(\lambda_{1}I_{m_1},\dots,\lambda_{r}I_{m_{r}})$ with $r\geq2$, $m_{1}+\cdots+m_{r}=n$, and $m_{1}\lambda_{1}+\cdots +m_{r}\lambda_{r}=0$. 
	
	Let $B$ be equal to $I_{n}$ except on its column 1 and on its row $m_{1}+1$. We define column 1 of $B$ exactly as in case 1, but we define row $m_{1}+1$ of $B$ as
	$$(-\lambda_{2},\underbrace{-1,\cdots,-1}_{m_{1}-1},1,\underbrace{0,\cdots,0}_{n-m_{1}-1}).$$
	
	By means of straightforward computations as in case 1 and using the fact that $\tr(J)=0$, it follows that the diagonal entries of $G=BHB^{-1}$ are all equal to zero. Moreover, the superdiagonal of $G$ is
	$$(1,\underbrace{-\lambda_{1},\cdots,-\lambda_{1}}_{m_{1}-1},
	\underbrace{-\lambda_{2},\cdots,-\lambda_{2}}_{m_{2}},\cdots,
	\underbrace{-\lambda_{r-1}, \cdots, -\lambda_{r-1}}_{m_{r-1}}, \underbrace{-\lambda_{r}\cdots,-\lambda_{r}}_{m_{r}-1}).$$
	However, the subdiagonal of $G$ depends on the value of $m_{1}$.
	\begin{itemize}
		\item If $m_{1}=1$, the subdiagonal of $G$ is $$(\lambda_{2}^{2},
		\underbrace{-\lambda_{2},\cdots,-\lambda_{2}}_{m_{2}-1},
		\underbrace{-\lambda_{3}, \cdots, -\lambda_{3}}_{m_{3}},\cdots,
		\underbrace{-\lambda_{r},\cdots,-\lambda_{r}}_{m_{r}}).$$
		\item If $m_{1}\geq 2$, the subdiagonal of $G$ is $$(\lambda_{1}^{2},\underbrace{-\lambda_{1},\cdots,-\lambda_{1}}_{m_{1}-1},
		\underbrace{-\lambda_{2},\cdots,-\lambda_{2}}_{m_{2}-1},
		\underbrace{-\lambda_{3}, \cdots, -\lambda_{3}}_{m_{3}},\cdots,
		\underbrace{-\lambda_{r},\cdots,-\lambda_{r}}_{m_{r}}).$$
	\end{itemize}
	Since $J$ is nonsingular, no eigenvalue of $J$ is zero. Hence the superdiagonal and subdiagonal entries of $G$ are nonzero. Thus, $H$ verifies the $(U,L)$-property.
	\end{proof}
	
	\begin{cor}\label{cor:diagonal} Let $J\in M_{n}(\mathbb{F})$ be any nonscalar diagonal Jordan matrix such that $tr(J)=0$. Then $J$ is the sum of two strictly $n$-zero matrices. 
	\end{cor}
	\begin{proof}
		By similarity, $J=\mathrm{diag}(\lambda_{1}I_{m_1},\cdots,\lambda_{r}I_{m_{r}})$, where $r\geq2$, $\lambda_{1},\dots,\lambda_{r}$ are distinct elements of $\F$, $m_{1}+\cdots +m_{r}=n$, and $m_{1}\lambda_{1}+\cdots +m_{r}\lambda_{r}=0$. If $J$ is nonsingular, the conclusion follows from Theorem \ref{thm:1} and (iv) of Proposition \ref{prop:1}. From now on we assume that $J$ is singular and that $m$ is the algebraic multiplicity of the eigenvalue zero. 
		
		If $m=1$, we may assume that $\lambda_1=0$ and that $m_{1}=1$. By repeating the same argument as in the proof of case 2 of Theorem \ref{thm:1}, one sees that $J$ verifies the $(U,L)$-property. Then the conclusion follows from (iv) of Proposition \ref{prop:1}.
		
		If $m>1$, we may assume that $J=\mathrm{diag}(J'',0I_{m-1})$, where $J''=\mathrm{diag}(0,J')$ and $J'$ is the nonsingular part of $J$. By the preceding paragraph, $J''$ verifies the $(U,L)$-property. Letting $J(0,m-1)=0I_{m-1}$ and $a=1\in\F$, it follows that $J(0,m-1)+aJ_{m-1}(0)=J_{m-1}(0)$. Since $J_{m-1}(0)$ is a strictly $(m-1)$-zero matrix, Theorem \ref{thm:jo} now yields the desired conclusion.	\end{proof}	
	
		\begin{cor}\label{cor:0} Let $\mathbb{F}$ a field and let $J\in M_{n}(\mathbb{F})$ be a nonsingular nondiagonal Jordan matrix with at least two distinct eigenvalues and such that $\tr(J)=0$. If there exists $t\in\mathbb{F}\setminus\{0\}$ such that the spectrum of $tJ$ contains neither $1$ nor $-1$, then $J$ verifies the $(U,L)$-property. 
	\end{cor}
	\begin{proof} Let $\{\lambda_{1},\dots,\lambda_{r}\}$ be the spectrum of $J$. The matrix $tJ$ is similar to the nonsingular nondiagonal Jordan matrix $\mathcal{J}$ whose spectrum is $\{t\lambda_{1},\dots,t\lambda_{r}\}$ and whose corresponding Jordan blocks have exactly the same size as the corresponding Jordan blocks of $J$. So $\mathcal{J}$ has at least two distinct eigenvalues, $ \tr(\mathcal{J})=0$, and the spectrum of $\mathcal{J}$ contains neither $1$ nor $-1$. By Theorem \ref{thm:1}, $\mathcal{J}$ verifies the $(U,L)$-property. Combining (ii) and (iii) of Proposition \ref{prop:1}, the conclusion follows.
	\end{proof}
	
	\begin{cor}\label{cor:1} Let $\mathbb{F}$ be an infinite field and let $J\in M_{n}(\mathbb{F})$ be a nonsingular Jordan matrix with at least two distinct eigenvalues and such that $\tr(J)=0$. Then $J$ verifies the $(U,L)$-property. 
	\end{cor} 
	\begin{proof}
		If $J$ satisfies the assumptions of Theorem \ref{thm:1}, the conclusion follows. Now suppose that $J$ satisfies the given assumptions, and additionally suppose that $J$ is nondiagonal and that its spectrum contains $1$ or $-1$. Because $\mathbb{F}$ is an infinite field, there exists $t\in \mathbb{F}\setminus\{0,\pm1\}$ such that the spectrum of $tJ$ contains neither $1$ nor $-1$. By Corollary \ref{cor:0}, the conclusion follows.
	\end{proof}
	
	\begin{cor}\label{cor:2} Let $\mathbb{F}$ be an algebraically closed field with $char(\mathbb{F})=0$. If $A\in M_{n}(\F)$ is a nonsingular matrix with $\tr(A)=0$, then $A$ verifies the $(U,L)$-property. 
	
	\end{cor}
	\begin{proof} Since $\F$ is algebraically closed, $\F$ is an infinite field and $A$ is similar to its Jordan canonical form $J_{A}$.  The given assumptions imply that $J_{A}$ is nonsingular and that $J_{A}$ has at least two distinct eigenvalues. By Corollary \ref{cor:1}, $J_{A}$ verifies the $(U,L)$-property. Then the conclusion follows from (iii) of Proposition \ref{prop:1}. 
	\end{proof}

	The next result is a stronger version of Fillmore's Theorem for matrices with zero trace over algebraically closed fields of characteristic zero.
	
	\begin{thm}\label{thm:2} Let $\mathbb{F}$ be an algebraically closed field with $char(\mathbb{F})=0$, and let $A\in M_{n}(\mathbb{F})$ such that $\tr(A)=0$. Then $A$ is the sum of two strictly $n$-zero matrices. \end{thm} 
	
	\begin{proof} It suffices to prove the statement for $J_{A}$, the Jordan canonical form of $A$. If $A$ is nonsingular, the conclusion follows from Corollary \ref{cor:2}. If $A$ is singular, then we may assume that $J_{A}=\mathrm{diag}(J,J(0,m))$, where $J$ is the nonsingular part of $J_{A}$ and $m\geq1$ is the algebraic multiplicity of the eigenvalue zero. Since $\tr(J)=\tr(A)=0$, by Corollary \ref{cor:2}, $J$ verifies the $(U,L)$-property. Because $\text{char}(\mathbb{F})\neq 2$, Corollary \ref{cor:jo} yields the desired conclusion.
	\end{proof}

	\begin{rem}\label{rem:3} If $\text{char}(\mathbb{F})\neq 0$, the conclusion of Theorem \ref{thm:2} is no longer valid. For instance, let $p$ be a prime number and consider $A=I_{p}$ over an algebraically closed field $\mathbb{F}$ with $\text{char}(\mathbb{F})=p$. It is clear that $\tr(A)=0$. If $A=B+C$, where $B$ and $C$ are both strictly $p$-zero matrices over $\mathbb{F}$, we would have a contradiction because the only eigenvalue of $A-B$ is $1$ and the only eigenvalue of $C$ is zero. \end{rem}
	
	\section*{Acknowledgement}
	The authors would like to express their sincere thanks to the referee for his/her careful reading of the manuscript and helpful suggestions.

\end{document}